\documentclass[11pt]{article}
\usepackage{bbm}
\usepackage[bbgreekl]{mathbbol}
\usepackage{amsfonts}

\usepackage[all]{xy}
\usepackage{hyperref}
\usepackage[a4paper,left=2.5cm,right=2.5cm,bottom=4cm,top=3.5cm]{geometry}
\usepackage{mathrsfs}
\usepackage{amsmath}
\usepackage{amssymb}
\usepackage{amsthm}
\usepackage{lscape}
\usepackage{array}

\theoremstyle{plain} \newtheorem{prop}{Proposition}[section]
\theoremstyle{plain} \newtheorem{theo}[prop]{Theorem}
\theoremstyle{plain} \newtheorem*{theon}{Theorem}
\theoremstyle{plain} \newtheorem{lem}[prop]{Lemma}
\theoremstyle{plain} 
\theoremstyle{plain} \newtheorem{cor}[prop]{Corollary}
\theoremstyle{definition} \newtheorem{defn}[prop]{Definition}
\theoremstyle{remark} \newtheorem{rem}[prop]{Remark}
\theoremstyle{remark} 

\numberwithin{equation}{section}

\newcommand{\be}{\mathbf{e}}
\newcommand{\bepsilon}{\pmb{\epsilon}}
\newcommand{\bj}{\mathbf{j}}
\newcommand{\blambda}{\pmb{\lambda}}
\newcommand{\bm}{\mathbf{m}}
\newcommand{\bSO}{\mathbf{SO}}
\newcommand{\bu}{\mathbf{u}}
\newcommand{\bv}{\mathbf{v}}
\newcommand{\bx}{\mathbf{x}}
\newcommand{\by}{\mathbf{y}}
\newcommand{\bz}{\mathbf{z}}
\newcommand{\bzero}{\mathbf{0}}

\newcommand{\cA}{\mathcal{A}}
\newcommand{\cC}{\mathcal{C}}
\newcommand{\cJ}{\mathcal{J}}
\newcommand{\cL}{\mathcal{L}}
\newcommand{\cO}{\mathcal{O}}
\newcommand{\cT}{\mathcal{T}}
\newcommand{\cU}{\mathcal{U}}
\newcommand{\cX}{\mathcal{X}}

\newcommand{\dA}{\mathbbm{A}}
\newcommand{\dC}{\mathbb{C}}
\newcommand{\df}{\mathbbm{f}}
\newcommand{\dg}{\mathbbm{g}}
\newcommand{\dG}{\mathbbm{G}}
\newcommand{\dmu}{\bbmu}
\newcommand{\domega}{\bbomega}
\newcommand{\dphi}{\bbphi}
\newcommand{\dQ}{\mathbb{Q}}
\newcommand{\dR}{\mathbb{R}}
\newcommand{\dZ}{\mathbb{Z}}

\newcommand{\fF}{\mathfrak{F}}
\newcommand{\fS}{\mathfrak{S}}

\newcommand{\ran}{\mathrm{an}}
\newcommand{\rcan}{\mathrm{can}}
\newcommand{\rd}{\mathrm{d}}
\newcommand{\rdim}{\mathrm{dim}}
\newcommand{\re}{\mathrm{e}}
\newcommand{\rfor}{\mathrm{for}}
\newcommand{\rH}{\mathrm{H}}
\newcommand{\rHess}{\mathrm{Hess}}
\newcommand{\rHom}{\mathrm{Hom}}
\newcommand{\rint}{\mathrm{int}}
\newcommand{\rpoly}{\mathrm{poly}}
\newcommand{\rrpoly}{\mathrm{rpoly}}
\newcommand{\rSpec}{\mathrm{Spec}}
\newcommand{\rSpf}{\mathrm{Spf}}
\newcommand{\rstar}{\mathrm{star}}
\newcommand{\rvol}{\mathrm{vol}}

\newcommand{\sC}{\mathscr{C}}
\newcommand{\sG}{\mathscr{G}}
\newcommand{\sL}{\mathscr{L}}
\newcommand{\sM}{\mathscr{M}}
\newcommand{\sX}{\mathscr{X}}

\begin{document}

\title{\textbf{A non-archimedean analogue of Calabi-Yau theorem for totally degenerate abelian varieties}}
\author{Yifeng Liu
\\ Columbia University}
\date{June 11, 2010}
\maketitle

\begin{abstract}
We show an example of a non-archimedean version of the Calabi-Yau
theorem in complex geometry. Precisely, we consider totally
degenerate abelian varieties and certain probability measures on
their associated analytic spaces in the sense of Berkovich.
\end{abstract}

{\small \tableofcontents}

\section{Introduction}
\label{sec_intro}

The theorem of Calabi-Yau is one of most important results in
complex geometry which has many applications (e.g. Yau's famous
paper \cite{Ya0}). In one version, it claims the following fact. Let
$M$ be a compact complex manifold with an ample line bundle $L$.
Then for any smooth positive measure $\mu$ on $M$ with
$\int_M\mu=\int_M c_1(L)^{\wedge\rdim M}$, there is a positive
metric $\|\;\|$ on $L$, unique up to a constant multiple, such that
$c_1(L,\|\;\|)^{\wedge\rdim M}=\mu$.

One would like to ask the similar question for a non-archimedean
field, for example, if we replace $\dC$ by $\dC_p$ and complex
manifolds by non-archimedean analytic spaces, or Berkovich spaces in
\cite{Be1}. Hence let $X$ be a smooth proper (strictly) analytic
space over $\dC_p$ with an ample line bundle $L$ (in particular, it
is algebrizable by GAGA). Given any integrable metric $\|\;\|$ (cf.
\cite{Zh2}) on $L$, although we don't have a nice analogue of
$(1,1)$-form for $c_1(L,\|\;\|)$ in the non-archimedean situation so
far, we can still talk about its top wedge, i.e.,
$c_1(L,\|\;\|)^{\wedge\rdim M}$ which is defined by Chambert-Loir in
\cite{Ch}. The top wedge is a measure on the underlying compact
(metrizable) topological space of $X$. If $\|\;\|$ is semi-positive
in the sense of \cite{Zh1, Zh2}, then $c_1(L,\|\;\|)^{\wedge\rdim
M}$ is a positive measure in the following sense: $\int_X
fc_1(L,\|\;\|)^{\wedge\rdim M}\geq0$ for any non-negative continuous
real function $f$ on $X$. Then analogous to the complex case, given
a positive measure $\mu$ on $X$ with $\int_X\mu=\deg_L(X)$, we can
ask the following two questions:
\begin{description}
  \item[(E)] Does it exist a semi-positive metric $\|\;\|$ on
  $L$ such that $c_1(L,\|\;\|)^{\wedge\rdim
  M}=\mu$?
  \item[(U)] If it does, is it unique up to a constant multiple?
\end{description}

The question (U) has been perfectly answered by Yuan and Zhang in
\cite{YZ}. There, they have already used this uniqueness result to
prove several exciting theorems in algebraic dynamic systems. The
answer to the question (E) is in fact {\em negative} in general
which is due to several reasons. We are still not clear about the
nature of this question. One possible reason is that the notion of
being positive in the metric side and measure side are not quite
compatible as in the complex case. Nevertheless, we would like to
give one example where the answer to (E) is positive, of course,
under certain restrictions.

The following is our situation. Let us consider a totally degenerate
abelian variety $A$, say over $\dC_p$, i.e., the associated analytic
space $A^{\ran}\cong\left(\dG_m^d\right)^{\ran}/M$ for a complete
lattice $M\in\dG_m^d(\dC_p)$ (cf. Section \ref{sec_mumford}). We
have an evaluation map
$\tau_A:A^{\ran}=\left(\dG_m^d\right)^{\ran}/M\rightarrow\dR^d/\Lambda$
with a complete lattice $\Lambda\subset\dR^d$, which is continuous
and surjective. This map has a continuous section
$i_A:\dR^d/\Lambda\hookrightarrow A^{\ran}$. In fact, $i_A$
identifies $\dR^d/\Lambda$ as a strong deformation retract of
$A^{\ran}$ for which $i_A\circ\tau_A=\Phi(\cdot,1)$ for a strong
retraction map $\Phi:A^{\ran}\times[0,1]\rightarrow A^{\ran}$. Then
we prove the following theorem which is a certain non-archimedean
analogue of Calabi-Yau theorem.

\begin{theon}[Theorem \ref{theo_ncy}]
Let $A$ be a $d$-dimensional totally degenerate abelian variety over
$k$ (e.g. $\dC_p$) and $L$ an ample line bundle on $A$. For any
measure $\mu=f\rd\bx$ of $\dR^d/\Lambda$ with $f$ a positive smooth
function and $\int_{\dR^d/\Lambda}\mu=\deg_L(A)$, there is a
semi-positive metric $\|\;\|$ on $L$, unique up to a constant
multiple, such that $c_1(\overline{L})^d=(i_A)_*\mu$, where $\rd\bx$
is the Lebesgue measure on $\dR^d/\Lambda$ and
$\overline{L}=(L,\|\;\|)$.
\end{theon}

The main ingredient of the proof is a limit formula for the measure
associated to certain integrable metrics, which is Theorem
\ref{theo_measure}. According to this formula, the existence part
accounts to consider the following question in differential
equation. We would like to mention it here since it is really
interesting that we end up with a (real) Monge-Amp\`{e}re equation
on a real torus, quite similar to the complex case, although we are
doing non-archimedean geometry. In fact, let $\dR^d/\Lambda$ be the
real torus as above with usual coordinate $x_1,...,x_d$ and Lebesgue
measure $\rd\bx=\rd x_1\cdots\rd x_d$. Let $(g_{ij})_{i,j=1,...,d}$
be a positive definite (real symmetric) matrix. Then for any smooth
real function $f$ on $\dR^d/\Lambda$ such that
$\int_{\dR^d/\Lambda}\re^f\rd\bx=1$, we will show that there exists
a unique smooth (real) function $\phi$ such that:
\begin{itemize}
  \item The matrix $\left(g_{ij}+\frac{\partial^2\phi}{\partial x_i\partial
  x_j}\right)$ is positive definite;
  \item $\int_{\dR^d/\Lambda}\phi\rd\bx=0$;
  \item It satisfies the following real Monge-Amp\`{e}re equation
   \begin{equation*}
   \det\left(g_{ij}+\frac{\partial^2\phi}{\partial x_i\partial
   x_j}\right)=\det\left(g_{ij}\right)\cdot\re^f.
   \end{equation*}
\end{itemize}
We will reduce it to the complex Monge-Amp\`{e}re equation through
an easy process. Then Yau's famous attack on this equation will
imply our theorem. Hence the same PDE problem solves this
non-archimedean Calabi-Yau theorem as well!\\

At last, we would like make a remark about the notations. We use
$|\;|$ for the non-archimedean norm; $\|\;\|$ for the metrics on
line bundles. But due to the conventions, we will also use $|\;|$
for the usual absolute value of real numbers and the total
measure; $\|\;\|$ for the Euclidean norm of $\dR^d$ when $d>1$.\\

{\small

{\em Acknowledgements.} The paper is motivated by the work of Xinyi
Yuan and Shou-Wu Zhang \cite{YZ} on the uniqueness part of the
Calabi-Yau theorem and Gubler \cite{Gu2} on the tropical geometry of
totally degenerate abelian varieties. The author would also like to
thank Xander Faber, Xinyi Yuan and Shou-Wu Zhang for useful
discussion.

}

\section{Mumford's construction}
\label{sec_mumford}

In this section, we briefly recall Mumford's construction of
(formal) models of totally degenerate abelian varieties in
\cite[\S6]{Mu2},
also see \cite[\S4, \S6]{Gu2}.\\

{\em Valuation map.} Let $k$ be the completion of the algebraic
closure of a $p$-adic local field, for example, $k=\dC_p$. Let
$|\;|$ be the norm on $k$ and its extended valuation fields,
$k^{\circ}$ the sub-ring of $k$ consisting of elements $x\in k$ with
$|x|\leq1$, $k^{\circ\circ}$ the maximal ideal of $k^{\circ}$
consisting of elements $x\in k$ with $|x|<1$ and
$\widetilde{k}=k^{\circ}/k^{\circ\circ}$ the residue field which is
algebraically closed. We fix a logarithm $\log$ such that
$\log|x|\in\dQ$ for all $x\in k$.

Fix a split torus $T=\dG_{m,k}^d$ of rank $d\geq1$ over $k$. We have
the following valuation map
 \begin{equation*}
 \tau:T^{\ran}=\left(\dG_{m,k}^d\right)^{\ran}\longrightarrow\dR^d,\qquad
 t\mapsto(-\log|T_1(t)|,...,-\log|T_d(t)|)
 \end{equation*}
where $\left(\dG_{m,k}^d\right)^{\ran}$ is the associated
$k$-analytic space (cf. \cite[\S3.4]{Be1}) and $T_i$ are coordinate
functions. It is surjective and continuous with respect to the
underlying topology of the analytic space and the usual topology of
$\dR^d$. And $\tau$ has a continuous section
 \begin{equation*}
 i:\dR^d\longrightarrow T^{\ran},\qquad
 \bx\mapsto\xi_{\bx}
 \end{equation*}
where $\xi_{\bx}$ is the Shilov boundary of the affinoid domain
$\tau^{-1}(\bx)$ (cf. \cite[Corollary 4.5]{Gu2}).

Now consider a totally degenerate abelian variety $A$ over $k$,
i.e., $A^{\ran}\cong T^{\ran}/M$ for a complete lattice $M\subset
T(k)$. By a complete lattice $M$, we mean that $M$ bijectively maps
to a rational complete lattice $\Lambda\subset\dR^d$ under $\tau$.
Here, we say a complete lattice is rational if it has a basis whose
coordinates are in $\dQ$. Hence we have the induced map
$\tau_A:A^{\ran}\rightarrow\dR^d/\Lambda$ and
$i_A:\dR^d/\Lambda\hookrightarrow A^{\ran}$. In fact, $i_A$
identifies $\dR^d/\Lambda$ with a strong deformation retract or a
skeleton of $A^{\ran}$ as in
\cite[\S6.5]{Be1}.\\

{\em Rational polytopes.} A compact subset $\Delta$ of $\dR^d$ is
called a {\em polytope} if it is an intersection of finitely many
half-spaces $\{\bx\in\dR^d\;|\;\bm_i\cdot\bx\geq c_i\}$. We say
$\Delta$ is {\em rational} if we can choose all $\bm_i\in\dZ^d$ and
$c_i\in\dQ$. The dimension $\rdim(\Delta)$ of $\Delta$ is its usual
topological dimension and we denote by $\rint(\Delta)$ the
topological interior of $\Delta$ in $\dR^d$. A {\em closed face} of
$\Delta$ is either $\Delta$ itself or $B\cap\Delta$ where $B$ is the
boundary of a half-space containing $\Delta$. It is obvious that a
closed face of a (rational) polytope is again a (rational) polytope.
An {\em open face} is a closed face without its properly contained
closed faces.

A {\em (rational) polytopal complex} $\cC$ in $\dR^d$ is a locally
finite set of (rational) polytopes such that (1) if $\Delta\in\cC$,
then all its closed faces are in $\cC$ and (2) if $\Delta,
\Delta'\in\cC$, then $\Delta\cap\Delta'$ is either empty or a closed
face of both $\Delta$ and $\Delta'$. The polytopes of dimension $0$
are called {\em vertices}. We say $\cC$ is a {\em (rational)
polytopal decomposition} of $S\subset\dR^d$ if $S$ is the union of
all polytopes in $\cC$. In particular, if $S=\dR^d$, we say $\cC$ is
a (rational) polytopal decomposition of $\dR^d$.

For a (rational) complete lattice $\Lambda\subset\dR^d$, we say
$\cC$ is {\em $\Lambda$-periodic} if $\Delta\in\cC$ implies
$\Delta+\blambda\in\cC$ for all $\blambda\in\Lambda$. A (rational)
polytopal decomposition $\cC_{\Lambda}$ of $\dR^d/\Lambda$ for a
(rational) complete lattice $\Lambda$ is a $\Lambda$-periodic
(rational) polytopal decomposition $\cC$ of $\dR^d$ such that
$\Delta$ maps bijectively to its image under the projection
$\dR^d\rightarrow\dR^d/\Lambda$ for all $\Delta\in\cC$. A polytope,
a closed face or an open face of $\cC_{\Lambda}$ is a
$\Lambda$-translation equivalence class of the corresponding object
of $\cC$.

A continuous real function $f$ on $\dR^d$ is called {\em (rational)
polytopal} if there is a (rational) polytopal decomposition $\cC$ of
$\dR^d$ such that $f$ restricted to all $\Delta\in\cC$ is affine
(and takes rational values on all vertices). We denote by
$\sC_{\rpoly}(\dR^d)$ ($\sC_{\rrpoly}(\dR^d)$) the space of
(rational) polytopal continuous functions on $\dR^d$. We have the
following simple lemma.

\begin{lem}\label{lem_poly}
Let $\Lambda$ be a (rational) complete lattice of $\dR^d$ and $f$ in
$\sC_{\rpoly}(\dR^d)$ ($\sC_{\rrpoly}(\dR^d)$) satisfying\\
(a) There exist affine functions $z_{\blambda}$ for all
$\blambda\in\Lambda$ satisfying
$f(\bx+\blambda)=f(\bx)+z_{\blambda}(\bx)$ for all $\blambda$ and
$\bx\in\dR^d$;\\
(b) If $\Delta$ is a maximal connected subset on which $f$ is
affine, then
$\Delta$ is a bounded convex subset of $\dR^d$.\\
Then $\Delta$ is a (rational) polytope and if $\cC$ is the polytopal
complex generated by all such $\Delta$ and their closed faces, then
$\cC$ is a $\Lambda$-periodic (rational) polytopal decomposition of
$\dR^d$.
\end{lem}
\begin{proof}
Since $\Delta$ is closed, hence compact by (b). Since $f$ is
(rational) polytopal, there is a finite (rational) polytopal
decomposition of $\Delta$, hence $\Delta$ itself is a (rational)
polytope since it is convex. Let $\cC$ be the (rational) polytopal
complex generated by all such $\Delta$ and their closed faces which
is obviously a decomposition of $\dR^d$. The $\Lambda$-periodicity
is impled by (a).\\
\end{proof}

{\em Formal models.} Let $X$ be a projective scheme over $k$, a {\em
$k^{\circ}$-model} $\sX$ of $X$ is a scheme projective and flat over
$\rSpec\:k^{\circ}$ whose generic fibre $\sX_{\eta}\cong X$. A {\em
formal $k^{\circ}$-model} $\cX$ of $X$ is an admissible formal
scheme over $\rSpf\:k^{\circ}$ whose generic fibre $\cX_{\eta}\cong
X^{\ran}$. We denote by $\widetilde{\sX}$ (resp. $\widetilde{\cX}$)
the special fibre of $\sX$ (resp. $\cX$) which is a proper scheme
over $\rSpec\:\widetilde{k}$. The following result is due to Mumford
in the case $A$ is the base change of an abelian variety over a
$p$-adic local field (cf. \cite[Corollary 6.6]{Mu2}) and generalized
by Gubler in general case (cf. \cite[Proposition 6.3]{Gu2}).

\begin{prop}\label{prop_mumford}
Given a rational polytopal decomposition $\cC_{\Lambda}$ of
$\dR^d/\Lambda$, we may associate a formal $k^{\circ}$-model $\cA$
of $A$ whose special fibre $\widetilde{\cA}$ is reduced and the
irreducible components $Y$ of $\widetilde{\cA}$ are toric varieties
and one-to-one correspond to the vertices $\bv$ of $\cC_{\Lambda}$
by $\bv=\tau_A(\xi_Y)$, where $\xi_Y\in A^{\ran}$ is the point
corresponding to $Y$. The formal scheme $\cA$ has a covering by
formal open affine sets $\cU_{\Delta}$ for $\Delta\in\cC_{\Gamma}$.
Moreover, if $A$ is the base change of an abelian variety over a
$p$-adic local field, then $\cA$ can be constructed as a
$k^{\circ}$-model.
\end{prop}

\section{Toric metrized line bundles}
\label{sec_metrized}

In this section, we briefly recall the theory of metrized line
bundles and their associated measure for general varieties. We
introduce line bundles and toric metrized line bundles on $A$. Then
we prove the main result identifying certain toric integrable
metrics.\\

{\em Metrized line bundles and measure.} The general theory of
metrized line bundles is developed in \cite{Zh1}, \cite{Zh2}, also
see \cite{Ch} and \cite{Gu1}. Let $X$ be a projective scheme over
$k$ and $L$ a line bundle over $X$. A {\em metric} $\|\;\|$ on $L$
is given that, for all open subset $U$ of $X^{\ran}$ and a section
$s\in\Gamma(U,L^{\ran})$, a continuous function
$\|s\|:U\rightarrow\dR_{\geq0}$ such that $\| fs\|=|f|\cdot\|s\|$
for all $f\in\Gamma(U,\cO_U)$ which is zero only if $s=0$.

We say a metric is {\em algebraic} if it is defined by a model
$(\sX,\sL)$ where $\sX$ is a $k^{\circ}$-model of $X$ and $\sL$ is a
line bundle on $\sX$ such that $\sL_{\eta}\cong L^e$ for some
integer $e\geq1$. A metric is {\em formal} if we replace the
$k^{\circ}$-model $\sX$ by a formal $k^{\circ}$-model $\cX$ and
$\sL$ by a formal line bundle $\cL$ on $\cX$ such that
$\cL_{\eta}\cong (L^e)^{\ran}$. In fact, all formal metrics are
algebraic. An algebraic (resp. formal) metric is called {\em
semi-positive} if the reduction $\widetilde{\sL}$ (resp.
$\widetilde{\cL}$) has non-negative degree on all curves inside
$\widetilde{\sX}$ (resp. $\widetilde{\cX}$). In general, a metric on
$L$ is called {\em semi-positive} if it is the uniform limit of
algebraic semi-positive metrics. A metrized line bundle is called
{\em integrable} if it is isomorphic to a quotient of two
semi-positive metrized line bundles.

Next we recall the construction of measure by Chambert-Loir in
\cite[\S2]{Ch}. For simplicity, we only recall the algebraic case
(which we only need for calculation later) and for general metric,
one need to pass to the limit which we refer to {\em loc. cit.} for
details. Let $X$ be as above of dimension $d\geq1$ and $L_i$
($i=1,...,d$) line bundles on it. We endow $L_i$ with an algebraic
measure $\|\;\|_i$ induced by $(\sX,\sL_i)$ with
$(\sL_i)_{\eta}\cong L_i^{e_i}$ on a common model $\sX$ which is
assumed to be normal. Let $Y_j$ be the reduced irreducible
components of $\widetilde{\sX}$ and $\xi_j$ the unique point in the
inverse image of the generic point of $Y_j$ under the reduction map
$\pi:X^{\ran}\rightarrow\widetilde{\sX}$. Then we define
 \begin{equation}\label{measure}
 c_1(\overline{L_1})\wedge\cdots\wedge c_1(\overline{L_d})=\frac{1}{e_1\cdots
 e_d}\sum_j m_j\left(c_1(\widetilde{\sL_1})\cdots c_1(\widetilde{\sL_d})|Y_j\right)\delta_{\xi_j}
 \end{equation}
where $\overline{L_i}=(L_i,\|\;\|_i)$, $m_j$ is the multiplicity of
$Y_j$ in $\widetilde{\sX}$ and $\delta_{\xi_j}$ is the normalized
Dirac measure supported at $\xi_j$. In general, the measure
$c_1(\overline{L_1})\wedge\cdots\wedge c_1(\overline{L_d})$ is
symmetric and $\dZ$-multi-linear and we have
 \begin{equation}\label{degree}
 \int_{X^{\ran}}c_1(\overline{L_1})\wedge\cdots\wedge
 c_1(\overline{L_d})=c_1(L_1)\cdots c_1(L_d)|X.
 \end{equation}\\

{\em Line bundles on $A$.} The theory of line bundles on totally
degenerate abelian varieties is very similar to that over complex
field. We refer to \cite[\S2]{BL} and \cite[Chapter 6]{FP} for more
details.

Let $A$ be a totally degenerate abelian variety as above and
$\check{M}=\rHom_k(T,\dG_{m,k})$ the character group of $T$. Let
$\check{T}$ be the split torus with character group $M$; i.e.,
$\check{T}=\rHom_k(M,\dG_{m,k})$. Then $\check{A}^{\ran}$ is
canonically isomorphic to $\check{T}^{\ran}/\check{M}$ where
$\check{A}$ is the dual abelian variety of $A$. Let $L$ be a line
bundle on $A$, the pull-back of $L$ to $T$ is trivial and is
identified with $T\times\dG_{a,k}$. Hence $L$ is identified with a
quotient $(T\times\dG_{a,k})/M$ whose action is given by an element
$\mu\mapsto Z_{\mu}$ of $\rH^1(M,\cO(T)^{\times})$. The function
$Z_{\mu}$ has the form $Z_{\mu}=d_{\mu}\sigma_{\mu}$ where
$d_{\mu}\in k^{\times}$, $\mu\mapsto\sigma_{\mu}$ is a group
homomorphism $\sigma:M\rightarrow\check{M}$ and
$d_{\mu\nu}d_{\mu}^{-1}d_{\nu}^{-1}=\sigma_{\nu}(\mu)$. By the
isomorphism $\tau:M\rightarrow\Lambda$, we get a unique symmetric
bilinear form $b$ on $\dR^d$ such that
$b(\tau(\mu),\tau(\nu))=-\log|\sigma_{\nu}(\mu)|$. Then $b$ is
positive definite if and only if $L$ is ample. And since
$\sigma_{\mu}$ is a character, $Z_{\mu}$ factors through $\tau$,
hence uniquely determines a function $z_{\blambda}$ on $\dR^d$ such
that $z_{\blambda}(\tau(t))=-\log|Z_{\mu}(t)|$ for all $\mu\in M$
and $t\in T$, where $\blambda=\tau(\mu)$. The function
$z_{\blambda}$ is affine with
 \begin{equation}
 z_{\blambda}(\bx)=z_{\blambda}(\bzero)+b(\bx,\blambda),\qquad\blambda\in\Lambda,\;\bx\in\dR^d.\\
 \end{equation}

Before stating the next lemma, we introduce some notations. We fix a
$\dZ$-basis $(\blambda_1,...,\blambda_d)$ of $\Lambda$ once for all
and let $\fF=\{\bx=x_1\blambda_1+\cdots+x_d\blambda_d\;|\;0\leq
x_i<1\}$ be a fundamental domain of $\Lambda$. The volume of the
closure $\overline{\fF}$ under the usual Lebesgue measure $\rd\bx$
of $\dR^d$ only depends on $\Lambda$ and will be denoted by
$\rvol(\Lambda)$. We define
$R_{\fF}=\max_{\bx,\bx'\in\overline{\fF}}\|\bx-\bx'\|$ and $r_{\fF}$
to be the maximal radius of balls contained in $\overline{\fF}$. We
denote by $S^{d-1}\subset\dR^d$ the standard unit ball, $\bSO_d$ the
special orthogonal group of $\dR^d$ and $\fS_d$ the group of
$d$-permutations.

Let $(\be_1,...,\be_d)$ be the standard base of the Euclidean space
$\dR^d$, $\sC^l(\dR^d)$ ($l\geq0$ and $\sC=\sC^0$) the space of real
functions whose $l$-th partial derivatives exist and are continuous,
$\sC^{\infty}(\dR^d)$ the space of real smooth functions and
$\sC^{k,\alpha}(\dR^d)$ ($k\geq0$, $\alpha\in[0,1)$) the spaces of
real functions whose $k$-th partial derivatives exist and are
H\"{o}lder continuous with exponential $\alpha$. We denote by
$\sC^l_{\geq0}(\dR^d)$ the subspace of functions non-negative
everywhere and $\sC^l_{>0}(\dR^d)$ that of functions positive
everywhere; similarly, we have $\sC^{\infty}_{\geq0}$,
$\sC^{\infty}_{>0}$, $\sC^{k,\alpha}_{\geq0}$ and
$\sC^{k,\alpha}_{>0}$. For a certain function $f$, we let
$f_i=\nabla_{\be_i}f$, $f_{ij}=\nabla_{\be_i}\nabla_{\be_j}f$ and
$f_{ijk}=\nabla_{\be_i}\nabla_{\be_j}\nabla_{\be_k}f$ if the
corresponding directional derivatives exist.

Let $q(\bx)=\frac{1}{2}b(\bx,\bx)$ be the associated quadratic form
and $H_q=d!\det\left(q_{ij}\right)_{i,j=1,...,d}$ which is a
constant. We have the following lemma.

\begin{lem}\label{lem_deg}
If $L$ is ample, then $\deg_L(A)=\rvol(\Lambda)H_q$.
\end{lem}
\begin{proof}
Consider the morphism $\phi_L:A\rightarrow \check{A}$ associated to
$L$. Its lifting $T\rightarrow\check{T}$ restricts to
$\mu:M\rightarrow\check{M}$ on $M$. It is easy to see that $\mu$ is
injective since $L$ is ample and
 \begin{equation*}
 \deg(\mu)=[\check{M}:\mu(M)]=\rvol(\Lambda)^{-1}\det\left(b(\blambda_i,\blambda_j)\right)_{i,j=1,...,d}
 =\rvol(\Lambda)\det\left(q_{ij}\right)_{i,j=1,...,d}.
 \end{equation*}
By \cite[Theorem 6.15]{BL}, $\deg(\phi_L)=\deg(\mu)^2$ and by the
Riemann-Roch theorem \cite[\S16]{Mu1},
$(\deg_L(A)/d!)^2=\deg(\phi_L)$. Since $L$ is ample, $\deg_L(A)>0$
and hence equals $\rvol(\Lambda)H_q$.\\
\end{proof}

{\em Toric metrized line bundles.} The following proposition is due
to Gubler.

\begin{prop}\label{prop_metric}
Let $L=(T\times\dG_{a,k})/M$ be a line bundle on $A$ given by a
cocycle $(Z_{\mu})_{\mu\in M}$ as above. Let $\cA$ be a formal
$k^{\circ}$-model determined by a rational polytopal decomposition
$\cC_{\Lambda}$ of $\dR^d/\Lambda$ given by $\cC$.\\
(a) There is a one-to-one correspondence between all formal metrics
of $L$, i.e., formal model $\cL$ of $L^e$ (with $e$ minimal) on
$\cA$, with trivialization $(\cU_{\Delta})_{\Delta\in\cC_{\Lambda}}$
and functions $g\in\sC_{\rrpoly}(\dR^d)$ satisfying
 \begin{equation}\label{green}
 g(\bx+\blambda)=g(\bx)+z_{\blambda}(\bx);\qquad\blambda\in\Lambda,\;\bx\in\dR^d.
 \end{equation}
Moreover, if we denote by $\|\;\|$ the corresponding formal metric
on $L$, then we have
 \begin{equation}\label{metric}
 g\circ\tau=-\log\circ p^{*}\|1\|
 \end{equation}
 on $T$, where $p:\dR^d\rightarrow\dR^d/\Lambda$ is the
 projection.\\
(b) The reduction $\widetilde{\cL}$ is ample if and only if $g$ is
strongly polytopal convex with respect to $\cC$, i.e., it is convex
and the maximal connected subsets on which $g$ is affine are
$\Delta\in\cC$ with $\rdim(\Delta)=d$.
\end{prop}
\begin{proof}
For (a), by \eqref{green} and the fact that there exist
$\bm_{\blambda}\in\dZ^d$ for all $\blambda\in\Lambda$ such that
$b(\bx,\blambda)=\bm_{\blambda}\cdot\bx$, we can find a smallest
integer $e\geq1$ such that $e\cdot g$ has integer gradient
everywhere. Then by \cite[Proposition 6.6]{Gu2}, $e\cdot g$
determines a formal model $\cL$ of $L^e$. Hence $g$ determines a
formal metric on $L$. Then last identity \eqref{metric} follows from
{\em loc. cit.}

For (b), it follows from \cite[Corollary 6.7]{Gu2}.
\end{proof}

\begin{defn}\label{defn_toric}
Given $L$ as above, we call a metric determined by the above
proposition a {\em toric formal metric} and the corresponding
function $g$ the associated {\em formal Green function}. We denote
by $\sG_{\rfor}(L)$ the set of all formal Green functions of $L$ and
$\sG_+(L)$ the set of all uniform limits of formal Green functions
of $L$ associated to semi-positive toric formal metrics which we
call {\em semi-positive Green functions}. It is easy to see that
$g\in\sG_+(L)$ also satisfies \eqref{green} and the metric
determined by \eqref{metric} is semi-positive in its original sense.
Similarly, we denote by $\sG_{\rint}(L)$ the set of deference of
functions in $\sG_+(L')$ and $\sG_+(L'')$ with
$L=L'\otimes(L'')^{-1}$ which we call {\em integrable Green
functions}, hence the corresponding metric is integrable. Moreover,
the set of integrable Green functions for all line bundles $L$ over
$A$: $\sG_{\rint}(A)=\bigcup\sG_{\rint}(L)$ is a torsion-free
$\dZ$-module. All metrized line bundles corresponding to the Green
functions in $\sG_{\rint}(A)$ all called {\em toric metrized line
bundles}. At last we denote by $\sG(L)$ the set of continuous real
functions satisfying \eqref{green} and we simply call them {\em
Green functions} of $L$.
\end{defn}

The following proposition provides a certain large class of
semi-positive Green functions for an ample line bundle.

\begin{prop}
Let $L$ be an ample line bundle and $g\in\sG(L)\cap\sC^2(\dR^d)$
such that the matrix $(g_{ij}(\bx))_{i,j=1,...,d}$ is semi-positive
definite everywhere, then $g$ is inside $\sG_+(L)$.
\end{prop}

\begin{proof} The proof is divided into several steps.\\

{\em Step 1.} We reduce to the case where
$(g_{ij}(\bx))_{i,j=1,...,d}$ is positive definite for all
$\bx\in\dR^d$.

First, there exists a function $g_{\rcan}\in\sG(L)\cap\sC^2(\dR^d)$
such that the matrix $\left((g_{\rcan})_{ij}(\bx)\right)$ is
positive definite. By \cite[Lemma 6.5.2 (4)]{FP}, there is a group
homomorphism $c:\Lambda\rightarrow\dQ$ such that
$z_{\blambda}(\bzero)=q(\blambda)+c(\blambda)$. We linearly extend
$c$ to $\dR^d$ and define $g_{\rcan}=q+c$, then
$g_{\rcan}\in\sG(L)\cap\sC^2(\dR^d)$ and
$\left((g_{\rcan})_{ij}(\bx)\right)$ is a constant positive definite
matrix. Next, for any $g\in\sG(L)\cap\sC^2(\dR^d)$, $g_{ij}$ is
$\Lambda$-periodic for any $(i,j)$ and $g-g_{\rcan}$ is a
$\Lambda$-periodic $\sC^2$-function. For any $g$ in the proposition,
let $f=g-g_{\rcan}$ and $g_t=g_{\rcan}+tf$ for $t\in[0,1]$, then
$g_t\rightarrow g_1=g$ when $t\rightarrow1$ and
$\left((g_t)_{ij}(\bx)\right)$ are positive definite for all $t<1$.
The claim follows.\\

{\em Step 2.} Now fix a function $g$ as in the proposition but with
the condition that $(g_{ij}(\bx))_{i,j=1,...,d}$ is positive
definite, we are going to construct a sequence of functions
$g_n\in\sG_{\rfor}(L)\cap\sG_+(L)$ approaching $g$. For any $\bu\in
S^{d-1}$, the function $\nabla_{\bu}\nabla_{\bu}g$ is
$\Lambda$-periodic and positive, hence there exist $0<h_g<H_g$ such
that $h_g<\nabla_{\bu}\nabla_{\bu}g(\bx)<H_g$ for any $\bu\in
S^{d-1}$ and $\bx\in\dR^d$.

Let $N$ be a sufficiently large integer, for
$\bj=(j_1,...,j_d)\in\left(\frac{1}{N}\dZ\right)^d$, we let
$\blambda_{\bj}=j_1\blambda_1+\cdots+j_d\blambda_d$. For each $\bj$
such that $\blambda_{\bj}\in\fF$, we choose a positive number
$\epsilon_N(\bj)<\frac{1}{N^2}$ such that
$g(\blambda_{\bj})-\epsilon_N(\bj)\in\dQ$ and a vector
$\bepsilon_N(\bj)$ such that $\|\bepsilon_N(\bj)\|<\frac{1}{N}$,
$\nabla g(\blambda_{\bj})-\bepsilon_N(\bj)\in\dQ^d$ and the graph of
the function
 \begin{equation*}
 g^{(N)}_{\blambda_{\bj}}(\bx)=\left(\nabla
 g(\blambda_{\bj})-\bepsilon_N(\bj)\right)(\bx-\blambda_{\bj})+g(\blambda_{\bj})-\epsilon_N(\bj)
 \end{equation*}
is below the graph of $g$ which is possible since $g$ is strictly
convex. For general $\bj$, we let $\bj_0$ be the unique element in
$\left(\frac{1}{N}\dZ\right)^d$ such that $\bj-\bj_0\in\dZ^d$ and
$\blambda_{\bj_0}\in\fF$. Then we define
 \begin{equation*}
 g^{(N)}_{\blambda_{\bj}}(\bx)=\left(\nabla
 g(\blambda_{\bj})-\bepsilon_N(\bj_0)\right)(\bx-\blambda_{\bj})+g(\blambda_{\bj})-\epsilon_N(\bj_0).
 \end{equation*}
By construction, $g^{(N)}_{\blambda_{\bj}}\in\sC_{\rrpoly}(\dR)$. We
need the following lemma.

 \begin{lem}\label{lem_sp1}
 For each $\bj\in\left(\frac{1}{N}\dZ\right)^d$ and $\blambda\in\Lambda$, we have
  \begin{equation*}
  g(\bx)-g^{(N)}_{\blambda_{\bj}}(\bx)=g(\bx+\blambda)-g^{(N)}_{\blambda_{\bj}+\blambda}(\bx+\blambda).
  \end{equation*}
 \end{lem}
 \begin{proof}
 By definition, we have
  \begin{equation*}
  \begin{split}
  &g^{(N)}_{\blambda_{\bj}+\blambda}(\bx+\blambda)-g^{(N)}_{\blambda_{\bj}}(\bx)\\
  =&\left(\nabla g(\blambda_{\bj}+\blambda)-\nabla
  g(\blambda_{\bj})\right)(\bx-\blambda_{\bj})+g(\blambda_{\bj}+\blambda)-g(\blambda_{\bj})\\
  =&\nabla
  z_{\blambda}(\blambda_{\bj})(\bx-\blambda_{\bj})+z_{\blambda}(\blambda_{\bj})\\
  =&z_{\blambda}(\bx)\\
  =&g(\bx+\blambda)-g(\bx)
  \end{split}
  \end{equation*}
 where the third equality is because $z_{\blambda}$ is affine.\\
 \end{proof}

{\em Step 3.} We define a function $g^{(N)}$ by
 \begin{equation*}
 g^{(N)}(\bx)=\sup\limits_{\bj}g^{(N)}_{\blambda_{\bj}}(\bx)
 \end{equation*}
which is less than $g(\bx)$. We have

 \begin{lem}\label{lem_sp2}
 For any compact subset $V\subset\dR^d$, there exists a finite subset
 $\cJ_V\subset\left(\frac{1}{N}\dZ\right)^d$ such that
  \begin{equation*}
  g^{(N)}(\bx)=\max\limits_{\bj\in\cJ_V}g^{(N)}_{\blambda_{\bj}}(\bx)
  \end{equation*}
 for all $\bx\in V$.
 \end{lem}
 \begin{proof}
 We only need to prove that for given $M\in\dR$, there are only
 finitely many $\bj$ such that $g^{(N)}_{\blambda_{\bj}}(\bx)\geq M$
 for some $\bx\in V$.
 For a given $\bj$, we try to give a lower bound for the difference
 $g(\bx)-g^{(N)}_{\blambda_{\bj}}(\bx)$. Let $\bu=\bx-\blambda_{\bj}$, by definition,
  \begin{equation*}
  \begin{split}
  &g(\bx)-g^{(N)}_{\blambda_{\bj}}(\bx)\\
  >&\left(g(\bx)-g^{(N)}_{\blambda_{\bj}}(\bx)\right)-\left(g(\blambda_{\bj})-g^{(N)}_{\blambda_{\bj}}(\blambda_{\bj})
  \right)\\
  =&\int_0^1\frac{\rd}{\rd
  t}\left(g(\blambda_{\bj}+t\bu)-g^{(N)}_{\blambda_{\bj}}(\blambda_{\bj}+t\bu)\right)\rd
  t\\
  =&\int_0^1\nabla_{\bu}g(\blambda_{\bj}+t\bu)\rd
  t-\nabla_{\bu}g(\blambda_{\bj})-\bepsilon_N(\bj_0)\cdot\bu\\
  =&\int_0^1\left(\nabla_{\bu}g(\blambda_{\bj})+
  \int_0^t\nabla_{\bu}\nabla_{\bu}g(\blambda_{\bj}+s\bu)\rd s\right)\rd
  t-\nabla_{\bu}g(\blambda_{\bj})-\bepsilon_N(\bj_0)\cdot\bu\\
  =&\int_0^1\int_0^t\nabla_{\bu}\nabla_{\bu}g(\blambda_{\bj}+s\bu)\rd
  s\rd t-\bepsilon_N(\bj_0)\cdot\bu\\
  >&\frac{h_g\|\bu\|^2}{2}-\|\bepsilon_N(\bj_0)\|\cdot\|\bu\|\\
  >&\frac{h_g\|\bu\|^2}{2}-\frac{\|\bu\|}{N}.
  \end{split}
  \end{equation*}

We see that there is $N_M>0$ such that $\|\blambda_{\bj}-\bx\|>N_M$
implies $g^{(N)}_{\blambda_{\bj}}(\bx)<M$. Hence the lemma follows.
\end{proof}

The above lemma implies that $g^{(N)}\in\sC_{\rrpoly}(\dR^d)$ and is
convex. On the other hand, if
$g^{(N)}(\bx)=g^{(N)}_{\blambda_{\bj}}(\bx)$ for some $\bj$, then
$g^{(N)}(\bx+\blambda)=g^{(N)}_{\blambda_{\bj}+\blambda}(\bx+\blambda)$
for all $\blambda\in\Lambda$ since by Lemma \ref{lem_sp1},
$g^{(N)}_{\blambda_{\bj'}}(\bx+\blambda)>g^{(N)}_{\blambda_{\bj}+\blambda}(\bx+\blambda)$
will imply that
$g^{(N)}_{\blambda_{\bj'}-\blambda}(\bx)>g^{(N)}_{\blambda_{\bj}}(\bx)$
which is a contradiction. Again by the same lemma, we conclude that
$g^{(N)}$ satisfies \eqref{green}, hence is inside $\sG(L)$.\\

{\em Step 4.} Before proving that $g^{(N)}$ is semi-positive formal,
we would like bound the difference of it and $g$.

\begin{lem}\label{lem_sp3}
For any $\bx\in\dR^d$, we have
 \begin{equation*}
 0\leq g(\bx)-g^{(N)}(\bx)<\frac{R_{\fF}^2\cdot
 H_g+2R_{\fF}+2}{2N^2}.
 \end{equation*}
\end{lem}
\begin{proof}
The proof follows the same line as in Lemma \ref{lem_sp2}. Hence we
have
 \begin{equation*}
 g(\bx)-g^{(N)}(\bx)<\frac{H_g}{2}\|\blambda_{\bj}-\bx\|^2+\frac{1}{N}\|\blambda_{\bj}-\bx\|+\frac{1}{N^2}
 \end{equation*}
for any $\bj$. In fact, we can choose $\bj$ such that
$\|\blambda_{\bj}-\bx\|\leq\frac{R_{\fF}}{N}$. Hence the lemma
follows.
\end{proof}

Conversely, we have the following lemma on the estimate of the
gradient which will be used also in the next section.

\begin{lem}\label{lem_sp4}
Let $f$ be any convex rational polytopal continuous function on
$\dR^d$. Suppose that $|f(\bx)-g(\bx)|<\epsilon$ for any
$\bx\in\dR^d$, then for any $d$-dimensional polytope $\Delta$ on
which $f$ is affine, $\|\bm_{\Delta}-\nabla
g(\bx_0)\|\leq2\sqrt{\epsilon H_g}$ for all $\bx_0\in\Delta$, where
$\bm_{\Delta}$ is the gradient of $f$ on $\Delta$.
\end{lem}
\begin{proof}
By continuity, we can assume that $\bx_0\in\rint(\Delta)$ and
$\bx_0=\bzero$. We only need to prove the lemma for $\tilde{f}$ and
$\tilde{g}$ where $\tilde{f}(\bx)=f(\bx)-\nabla
g(\bzero)\cdot(\bx)-g(\bzero)$ and $\tilde{g}(\bx)=g(\bx)-\nabla
g(\bzero)\cdot(\bx)-g(\bzero)$. In this case, we need to prove that
$\|\bm_{\Delta}\|<2\sqrt{\epsilon H_g}$. For any $\bu\in S^{d-1}$,
we assume that $\bm_{\Delta}\cdot\bu\geq0$; otherwise, we take
$-\bu$. For $t>0$, consider
\begin{equation*}
\begin{split}
&\left(\tilde{f}(t\bu)-\tilde{g}(t\bu)\right)-\left(\tilde{f}(\bzero)-\tilde{g}(\bzero)\right)\\
=&\int_0^t\left(\nabla_{\bu}\tilde{f}(s\bu)-\nabla_{\bu}\tilde{g}(s\bu)\right)\rd
s\\
\geq&\bm_{\Delta}\cdot\bu\;t-\int_0^t\int_0^s\nabla_{\bu}\nabla_{\bu}\tilde{g}(r\bu)\rd
r\rd s\\
>&\bm_{\Delta}\cdot\bu\;t-\frac{H_gt^2}{2}
\end{split}
\end{equation*}
where the first inequality is due to the assumption that $f$ is
convex. On the other hand, it is less than $2\epsilon$, hence we
have
 \begin{equation*}
 \frac{H_gt^2}{2}-\bm_{\Delta}\cdot\bu\;t+2\epsilon>0
 \end{equation*}
for all $t>0$. Hence $\bm_{\Delta}\cdot\bu<2\sqrt{\epsilon H_g}$ and
then $\|\bm_{\Delta}\|\leq2\sqrt{\epsilon H_g}$.
\end{proof}

The above lemma immediately implies the following
 \begin{lem}\label{lem_sp5}
 Let $f$ and $\Delta$ be as above, then for any $\bx,\bx'\in\Delta$,
 the distance $\|\bx-\bx'\|\leq\frac{4}{h_g}\sqrt{\epsilon H_g}$.
 \end{lem}

In particular, if we apply the above lemma to $g^{(N)}$, we see that
$\Delta$ is compact for any maximal connected subset $\Delta$ on
which $g^{(N)}$ is affine.\\

{\em Step 5.} We would like to apply Lemma \ref{lem_poly}. Hence we
need to show that $\Delta$ is convex. By construction, $g^{(N)}$
coincides with some $g^{(N)}_{\blambda_{\bj}}$ restricted to
$\Delta$. Suppose that there are $\bx_0,\bx_1\in\Delta$ and
$t\in(0,1)$ such that $\bx_t=t\bx_1+(1-t)\bx_0\not\in\Delta$, then
$g^{(N)}(\bx_t)=g^{(N)}_{\blambda_{\bj'}}(\bx_t)$ for some
$\bj'\neq\bj$. Again by construction,
$g^{(N)}_{\blambda_{\bj'}}(\bx_t)>g^{(N)}_{\blambda_{\bj}}(\bx_t)$.
Hence there is one point $\bx\in\{\bx_0,\bx_1\}$ such that
$g^{(N)}_{\blambda_{\bj'}}(\bx)>g^{(N)}_{\blambda_{\bj}}(\bx)$ which
is a contradiction. Now by Lemma \ref{lem_poly}, $g^{(N)}$
determines a $\Lambda$-periodic rational polytopal decomposition
$\cC$ of $\dR^d$.

Finally, we prove that for any $\Delta\in\cC$, $\Delta$ maps
bijectively to its image under the projection
$p:\dR^d\rightarrow\dR^d/\Lambda$ when $N$ is sufficiently large. We
can assume $\dim(\Delta)=d$. By Lemma \ref{lem_sp3} and
\ref{lem_sp5}, we see that this holds if
 \begin{equation*}
 \frac{4}{h_g}\sqrt{\frac{R_{\fF}^2\cdot H_g+2R_{\fF}+2}{2N^2}\cdot
 H_g}<2r_g.
 \end{equation*}
Now by Proposition \ref{prop_metric}, $g^{(N)}$ is a semi-positive
formal Green function for large $N$. Hence the proposition follows
by Lemma \ref{lem_sp3}.
\end{proof}

The proposition has the following direct corollary.

\begin{cor}
For any line bundle $L$ on $A$, if $g\in\sG(L)\cap\sC^2(\dR^d)$,
then $g$ is integrable, i.e., $g\in\sG_{\rint}(L)$.
\end{cor}

\section{A limit formula for the measure}
\label{sec_measure}

In this section, we prove a formula for the measure of metrics
determined by certain integrable Green functions.\\

{\em Measures on torus and mixed Hessian.} Recall that we have a
closed manifold $\dR^d/\Lambda$. Similar to $\dR^d$, we define the
space of real functions $\sC^l(\dR^d/\Lambda)$,
$\sC^{\infty}(\dR^d/\Lambda)$, $\sC^{k,\alpha}(\dR^d/\Lambda)$ and
also $\sC_{\geq0}^{?}(\dR^d/\Lambda)$,
$\sC_{>0}^{?}(\dR^d/\Lambda)$. A {\em measure} on $\dR^d/\Lambda$ is
a continuous linear functional
$\mu:\sC(\dR^d/\Lambda)\rightarrow\dR$; it is {\em semi-positive} if
$\mu(f)\geq0$ for all $f\in\sC_{\geq0}(\dR^d/\Lambda)$; {\em
positive} if $\mu(f)>0$ for $0\neq f\in\sC_{\geq0}(\dR^d/\Lambda)$.
The space of all measure (resp. semi-positive measure, positive
measure) is denoted by $\sM(\dR^d/\Lambda)$ (resp.
$\sM_{\geq0}(\dR^d/\Lambda)$, $\sM_{>0}(\dR^d/\Lambda)$). It is
endowed with the weak topology; i.e., a sequence
$\mu_n\rightarrow\mu$ if and only if $\mu_n(f)\rightarrow\mu(f)$ for
all $f\in\sC(\dR^d/\Lambda)$. Recall that we have the Lebesgue
measure $\rd\bx$ on $\dR^d/\Lambda$, hence the spaces of functions
$\sC^{?}(\dR^d/\Lambda)$ can be identified as a subset of measure by
integration which we denote by $\sM^{?}(\dR^d/\Lambda)$ for
$?=l;\;\infty;\;k,\alpha$. Under this identification, being
semi-positive (non-negative) or positive for a function coincides
with that for a measure. Hence we also introduce the notation
$\sM_{\geq0}^?(\dR^d/\Lambda)$ or $\sM_{>0}^?(\dR^d/\Lambda)$ for
their obvious meaning. We write $\mu\leq\mu'$ if
$\mu'-\mu\in\sM_{\geq0}(\dR^d/\Lambda)$. Finally, we denote by
$|\mu|$ the totally measure of $\mu$, i.e., $|\mu|=\mu(1)$. It is
easy to see that if $\mu$ is semi-positive and $|\mu|=0$, then
$\mu=0$.

\begin{defn}
For functions $g_1,...,g_d\in\sC^2(\dR^d)$, we define the so-called
{\em mixed Hessian} of $g_1,...,g_d$ to be the function
 \begin{equation*}
 \rHess_{g_1,...,g_d}(\bx)=\sum_{\mu,\nu\in\fS_d}(-1)^{\epsilon(\mu\nu)}\prod_{i=1}^d(g_i)_{\mu(i)\nu(i)}(\bx)
 \end{equation*}
where $\epsilon(\mu)=1$ (resp. $-1$) if $\mu$ is an even (resp. odd)
permutation. In particular, when $g_1=\cdots=g_d=g$,
 \begin{equation*}
 \rHess_g(\bx):=\rHess_{g,...,g}(\bx)=d!\det(g_{ij})_{i,j=1,...,d}
 \end{equation*}
is $d!$ times the determinant of the usual Hessian matrix of $g$.
\end{defn}

The following lemma is an easy calculus exercise.

\begin{lem}\label{lem_hess}
Let $L$ be an ample line bundle on $A$ and let
$g_1,...,g_d\in\sG(L)\cap\sC^3(\dR^d)$, then $\rHess_{g_1,...,g_2}$
is $\Lambda$-periodic and hence in $\sC^1(\dR^d/\Lambda)$. We have
 \begin{equation*}
 \int_{\dR^d/\Lambda}\rHess_{g_1,...,g_d}(\bx)\rd\bx=\deg_L(A).
 \end{equation*}
\end{lem}

\begin{proof}
First, we check the case $g_1=\cdots=g_d=g_{\rcan}$. Then
 \begin{equation*}
 \int_{\dR^d/\Lambda}\rHess_{g_{\rcan}}(\bx)\rd\bx=\int_{\dR^d/\Lambda}H_q\rd\bx=\rvol(\Lambda)H_q
 \end{equation*}
which equals $\deg_L(A)$ by Lemma \ref{lem_deg}. In general, since
any two $g,g'\in\sG(L)\cap\sC^3(\dR^d)$ differ by a
$\Lambda$-periodic function $f\in\sC^3(\dR^d)$, by multi-linearity
and symmetry, we only need to prove that for such $f$ and
$g_2,...,g_d\in\sG(L)\cap\sC^3(\dR^d)$,
 \begin{equation}\label{hess_1}
 \int_{\dR^d/\Lambda}\rHess_{f,g_2,...,g_d}(\bx)\rd\bx=0.
 \end{equation}
By a linear change of coordinates, one can assume that
$\blambda_i=\be_i$ for $i=1,...,d$. Hence \eqref{hess_1} becomes
 \begin{equation*}
 \int_{[0,1]^d}\rHess_{f,g_2,...,g_d}(\bx)\rd x_1\cdots\rd x_d=0.
 \end{equation*}
For each $i=1,...,d$, let us denote by
 \begin{equation*}
 F^+_i=\{\bx=(x_1,...,x_d)\in [0,1]^d\;|\;x_i=1\};\quad
 F^-_i=\{\bx=(x_1,...,x_d)\in [0,1]^d\;|\;x_i=0\}.
 \end{equation*}
Then
 \begin{equation*}
 \begin{split}
 &\int_{[0,1]^d}\rHess_{f,g_2,...,g_d}(\bx)\rd x_1\cdots\rd x_d\\
 =&\int_{[0,1]^d}\sum_{\mu,\nu\in\fS_d}(-1)^{\epsilon(\mu\nu)}f_{\mu(1)\nu(1)}\prod_{i=2}^d(g_i)_{\mu(i)\nu(i)}
 \rd x_1\cdots\rd x_d\\
 =&\sum_{\mu,\nu\in\fS_d}(-1)^{\epsilon(\mu\nu)}
 \left(\int_{F_{\mu(1)}^+}f_{\nu(1)}\prod_{i=2}^d(g_i)_{\mu(i)\nu(i)}\rd\by-
 \int_{F_{\mu(1)}^-}f_{\nu(1)}\prod_{i=2}^d(g_i)_{\mu(i)\nu(i)}\rd\by-S_{\mu,\nu}
 \right)\\
 =&-\sum_{\mu,\nu\in\fS_d}(-1)^{\epsilon(\mu\nu)}S_{\mu,\nu}
 \end{split}
 \end{equation*}
 since $f$ and $(g_i)_{\mu(i)\nu(i)}$ are $\Lambda$-periodic, where
 \begin{equation*}
 S_{\mu,\nu}=\int_{[0,1]^d}f_{\nu(1)}\sum_{j=2}^d(g_j)_{\mu(1)\mu(j)\nu(j)}\prod_{\substack{i=2\\i\neq
 j}}^d(g_i)_{\mu(i)\nu(i)}\rd x_1\cdots\rd x_d.
 \end{equation*}
But then
 \begin{equation*}
 \begin{split}
 &\sum_{\mu,\nu\in\fS_d}(-1)^{\epsilon(\mu\nu)}S_{\mu,\nu}\\
 =&\sum_{j=2}^d\sum_{\nu\in\fS_d}(-1)^{\epsilon(\nu)}\int_{[0,1]^d}f_{\nu(1)}
 \left(\sum_{\mu\in\fS_d}(-1)^{\epsilon(\mu)}(g_j)_{\mu(1)\mu(j)\nu(j)}
 \prod_{\substack{i=2\\i\neq j}}^d(g_i)_{\mu(i)\nu(i)}\right)\rd x_1\cdots\rd
 x_d\\
 =&0
 \end{split}
 \end{equation*}
 since in the inner summation, the terms with $\mu$ and $\mu'$ such
 that $\mu'(1)=\mu(j)$, $\mu'(j)=\mu(1)$ and $\mu'(i)=\mu(i)$ for
 $i\neq1,j$ cancel each other. Hence the lemma follows.\\
\end{proof}

{\em Limit formula for the measure.} Before we state the formula, we
would like to introduce the notion of dual polytopes. The main
reference for this is \cite[\S2, \S3]{Mc}.  Let $\cC$ be a
(rational) polytopal decomposition of $\dR^d$ and $f$ be a
(rational) polytopal function, strongly polytopal convex with
respect to $\cC$. Then for any vertex $\bv\in\cC$, let $\rstar(\bv)$
be the set of all $d$-dimensional polytopes $\Delta\in\cC$
containing $\bv$ and for such $\Delta$, let $\bm_{\Delta}$ be the
gradient of $f$ on $\Delta$ (which is just the {\em peg} in {\em
loc. cit.}). Then the {\em dual polytopal} of $\bv$ with respect to
$f$ is defined to be the convex hull of points $\bm_{\Delta}$ for
all $\Delta\in\rstar(\bv)$, which we denote by $\widehat{\bv}^f$. It
is a $d$-dimensional (rational) polytope.

Now we are going to prove the following main theorem of this
section. Recall that we have the embedding
$i_A:\dR^d/\Lambda\hookrightarrow A^{\ran}$.

\begin{theo}\label{theo_measure}
Let $\overline{L_i}=(L_i,\|\;\|_i)$ ($i=1,...,d$) be $d$ integrable
metrized line bundles on $A$, where $\|\;\|_i$ are toric determined
by Green functions $g_i\in\sG_{\rint}(L_i)\cap\sC^3(\dR^d)$. Then we
have the following equality of measures on $A^{\ran}$:
 \begin{equation*}
 c_1(\overline{L_1})\wedge\cdots\wedge
 c_1(\overline{L_d})=(i_A)_*\rHess_{g_1,...,g_d}\rd\bx.
 \end{equation*}
\end{theo}
\begin{proof}
The proof is divided into several steps.\\

{\em Step 1.} First we reduce to the case
$\overline{L_1}=\cdots=\overline{L_d}=\overline{L}=(L,\|\;\|)$ where
$L$ is an ample line bundle on $A$ and $\|\;\|$ is determined by a
Green function $g\in\sG_+(L)\cap\sC^3(\dR^d)$ such that the matrix
$(g_{ij})_{i,j=1,...,d}$ is positive-definite everywhere. Assuming
this, consider the subset $\sG'(A)\subset\sG_{\rint}(A)$ consisting
of such Green functions. Then $g,g'\in\sG'(A)$ implies
$ag+bg'\in\sG'(A)$ for $(a,b)\in\dZ_{\geq0}^2-\{(0,0)\}$. For any
continuous function $f$ on $A^{\ran}$, consider the functional
 \begin{equation*}
 \ell_f(g_1,...,g_d)=\left(c_1(\overline{L_1})\wedge\cdots\wedge
 c_1(\overline{L_d})-(i_A)_*\rHess_{g_1,...,g_d}\rd\bx\right)(f)
 \end{equation*}
which is symmetric and $\dZ$-multi-linear in $g_1,...,g_d$ and
$\ell_f(g,...,g)=0$ for $g\in\sG'(A)$ by our assumption. Then for
$g_1,...,g_d\in\sG'(A)$ and $t_1,...,t_d\in\dZ^d_{>0}$,
 \begin{equation*}
 \begin{split}
 0&=\ell_f\left(\sum_{i=1}^d t_ig_i,...,\sum_{i=1}^d t_ig_i\right)\\
 &=\sum_{\substack{k_1,...,k_d\geq0\\k_1+\cdots+k_d=d}}\frac{d!}{k_1!\cdots k_d!}\ell_f(...,g_i,...,g_i,...)
 t_1^{k_1}\cdots t_d^{k_d}
 \end{split}
 \end{equation*}
where $g_i$ appears $d_i$ times in the second $\ell_f$. Hence
$\ell_f(...,g_i,...,g_i,...)=0$. In particular,
$\ell_f(g_1,...,g_d)=0$ for all $g_i\in\sG'(A)$. But on the other
hand, $\sG'(A)$ generates the whose space
$\sG_{\rint}(A)\cap\sC^3(\dR^d)$ by definition and the existence and
smoothness of $g_{\rcan}$. Then $\ell_f(g_1,...,g_d)=0$ for all
$g_i\in\sG_{\rint}(A)\cap\sC^3(\dR^d)$. The theorem follows.\\

{\em Step 2.} Now fix $g\in\sG'(A)$ as above and assume
$g\in\sG_+(L)$ for an ample line bundle $L$. By Proposition
\ref{prop_metric}, there are $g_n\in\sG_{\rfor}(L)$ such that
$g_n\rightarrow g$ and the corresponding formal $k^{\circ}$-models
$(\cX_n,\cL_n)$ satisfying that $\widetilde{\cL_n}$ is ample on
$\widetilde{\cX_n}$. Hence the models are in fact algebrizable. Now
we view $\cX_n$ as schemes projective and flat over
$\rSpec\;k^{\circ}$ and $(\cL_n)_{\eta}\cong L^{e_n}$. We denote the
corresponding metrized line bundle determined by $g_n$ by
$\overline{L}_n=(L,\|\;\|_n)$ and by $g$ by
$\overline{L}=(L,\|\;\|)$. By \eqref{measure} and Proposition
\ref{prop_mumford}, the measure $c_1(\overline{L}_n)^{\wedge d}$ is
supported on $i_A(\dR^d/\Lambda)$ hence also for their limit. If we
let $\mu_n$ be the measure $c_1(\overline{L}_n)^{\wedge d}$
restricted on $\dR^d/\Lambda$ and $\mu=\lim_n\mu_n$. Then we only
need to prove that $\mu=\mu_g$ as elements in $\sM(\dR^d/\Lambda)$,
where $\mu_g$ is the measure
$\rHess_g\rd\bx\in\sM^3_{>0}(\dR^d/\Lambda)$.

We claim that for any $\delta>0$, we have
 \begin{equation}\label{claim}
 \mu\leq(1+\delta)\mu_g.
 \end{equation}
Assuming this, then $\mu\leq\mu_g$. But by \eqref{degree} and Lemma
\ref{lem_hess}, $|\mu|=|\mu_g|$; i.e., $|\mu_g-\mu|=0$. Hence
$\mu=\mu_g$ which confirms the theorem.\\

{\em Step 3.} The last step is dedicated to prove the above claim
\eqref{claim}. We prove that for any $\delta>0$ and
$f\in\sC_{\geq0}(\dR^d/\Lambda)$, $\mu(f)\leq(1+\delta)\mu_g(f)$.
Given $\epsilon_1>0$, take $\overline{L}_n$ as above such that
$|g_n-g|<\epsilon_1$, then the model $\cX_n$ determines a rational
polytopal decomposition $(\cC_n)_{\Lambda}$ of $\dR^d/\Lambda$ which
comes from a $\Lambda$-periodic rational polytopal decomposition
$\cC_n$ of $\dR^d$. And the function $g_n$ is rational polytopal and
strictly polytopal convex with respect to $\cC_n$. By Proposition
\ref{prop_mumford}, the set of irreducible components of
$\widetilde{\cX_n}$ is identified with the set of
$\Lambda$-translation classes of vertices in $\cC_n$. Hence if we
denote $Y_{\bv}$ the irreducible component corresponding to $\bv$,
then $Y_{\bv}=Y_{\bv'}$ if and only if $\bv=\bv'+\blambda$ for some
$\blambda\in\Lambda$. By \eqref{measure}, we have
 \begin{equation}\label{mu_n}
 \mu_n=\frac{1}{e_n^d}\sum_{\bv\in\fF}\deg_{\overline{\cL_n}}(Y_{\bv})\delta_{p(\bv)}
 \end{equation}
where we recall that $\fF\subset\dR^d$ is the fixed fundamental
domain of $\Lambda$ and $p:\dR^d\rightarrow\dR^d/\Lambda$ is the
projection. We recall a formula in \cite[p.366 (36)]{Gu2} which is
deduced from \cite[p.112 Corollary]{Fu}, that
 \begin{equation*}
 \deg_{\overline{\cL_n}}(Y_{\bv})=d!\cdot\rvol(\widehat{\bv}^{e_ng_n}).
 \end{equation*}
Hence we have
 \begin{equation}
 \eqref{mu_n}=d!\sum_{\bv\in\fF}\rvol(\widehat{\bv}^{g_n})\delta_{p(\bv)}.
 \end{equation}
For any integer $N>0$, we divide $\fF$ into $N^d$ blocks as follows.
For $(b_1,...,b_d)\in\{0,1,...,N-1\}^d$, let
 \begin{equation*}
 \fF^{(N)}_{b_1,...,b_d}=\left\{\bx=x_1\blambda_1+\cdots+x_d\blambda_d\;|\;
 \frac{b_i}{N}\leq x_i<\frac{b_i+1}{N}\right\}.
 \end{equation*}
Then $\fF=\bigsqcup\fF^{(N)}_{b_1,...,b_d}$ and
$\overline{\fF^{(N)}_{b_1,...,b_d}}$ is a $d$-dimensional rational
polytope. For any $\epsilon_2>0$, there exists $N(\epsilon_2)>0$
such that for any $N\geq N(\epsilon_2)$,
 \begin{equation}\label{epsilon_2}
 \max_{\bx\in\overline{\fF^{(N)}_{b_1,...,b_d}}}g_{ij}(\bx)-
 \min_{\bx\in\overline{\fF^{(N)}_{b_1,...,b_d}}}g_{ij}(\bx)<\epsilon_2
 \end{equation}
for all $i,j=1,...,d$ and $(b_1,...,b_d)$. We now assume $N$ is that
large and consider on a block, say without lost of generality,
$\fF^{(N)}=\fF^{(N)}_{0,...,0}$. Then
 \begin{equation*}
 \mu_n|_{\fF^{(N)}}(f)=d!\sum_{\bv\in\fF^{(N)}}\rvol(\widehat{\bv}^{g_n})f(\bv)\leq
 d!\cdot\rvol\left(\Delta^{(N)}\right)\cdot\sup_{\bx\in\fF^{(N)}}f(\bx)
 \end{equation*}
where $\Delta^{(N)}$ is the convex hull of $\bm_{\Delta}$ (pegs
induced by $g_n$) for (finitely many) $d$-dimensional polytopes
$\Delta\in\cC_n$ such that $\Delta\cap\fF^{(N)}\neq\emptyset$. Now
we are going to give an upper bound for this volume. Let $\bx_0$ be
the point $\frac{1}{2N}(\blambda_1+\cdots+\blambda_d)$ which is the
center of symmetry of $\overline{\fF^{(N)}}$. The volume
$\rvol\left(\Delta^{(N)}\right)$ will keep unchanged under following
operations:
\begin{description}
  \item[a)] We replace $g_n$ by $\tilde{g}_n$ where
 \begin{equation*}
 \tilde{g}_n(\bx)=g_n(\bx)-\nabla g(\bx_0)\cdot(\bx-\bx_0)-g(\bx_0);
 \end{equation*}
(we also let $\tilde{g}(\bx)=g(\bx)-\nabla
g(\bx_0)\cdot(\bx-\bx_0)-g(\bx_0)$.)
  \item[b)] We make a translation $\bx'=\bx-\bx_0$;
  \item[c)] We apply a rotation $\bx''=R.\bx'$ for some $R\in\bSO_d$.
\end{description}

Hence we may assume that

\begin{description}
  \item[a')] $\nabla g(\bzero)=\bzero$ and $|g_n-g|<\epsilon_1$;
  \item[b')] $g(\bzero)=0$;
  \item[c')] \begin{equation*}
 (g_{ij}(\bzero))=\left(
                    \begin{array}{cccc}
                      h_{11} &  &  &  \\
                       & h_{22} &  &  \\
                       &  & \ddots &  \\
                       &  &  & h_{dd} \\
                    \end{array}
                  \right)
 \end{equation*}
with $h_g<h_{11}\leq\cdots\leq h_{dd}<H_g$;
  \item[d')] \begin{equation*}
 \fF^{(N)}=\left\{\bx=x_1\blambda'_1+\cdots+x_d\blambda'_d\;|\;-\frac{1}{2N}\leq
 x_i<\frac{1}{2N}\right\}
 \end{equation*}
where $\blambda'_i=R.\blambda_i$ for certain $R\in\bSO_d$.
\end{description}

For any $\epsilon_3\geq0$, we also introduce the following
 \begin{equation*}
 \fF^{(N)}_{g,\epsilon_3}=\left\{\bx'=(1+\epsilon_3)h_{11}x_1\be_1+\cdots+(1+\epsilon_3)h_{dd}x_d\be_d
 \;|\;\bx=x_1\be_1+\cdots+x_d\be_d\in\fF^{(N)}\right\}.
 \end{equation*}
The following lemma is obvious.

\begin{lem}\label{lem_ball}
For any $\bx'\in\fF^{(N)}_{g,0}$, the ball
$B\left(\bx',\frac{\epsilon_3h_gr_{\fF}}{N}\right)$ is contained
in $\fF^{(N)}_{g,\epsilon_3}$.\\
\end{lem}

Now for any $\bx\in\fF^{(N)}$, we have
 \begin{equation*}
 g_i(\bx)=\int_0^1\nabla_{\bx}\nabla_{\be_i}g(t\bx)\rd t.
 \end{equation*}
By \eqref{epsilon_2}, we have
 \begin{equation*}
 |g_i(\bx)-h_{ii}x_i|\leq\epsilon_2(|x_1|+\cdots+|x_d|)\leq\frac{\epsilon_2\cdot
 dR_{\fF}}{2N}.
 \end{equation*}
The point $\bx'=h_{11}x_1\be_1+\cdots+h_{dd}x_d\be_d$ is in
$\fF^{(N)}_{g,0}$. Let $\bm_{\Delta}(\bx)$ be (any) peg of $\Delta$
containing $\bx$, then we have, by Lemma \ref{lem_sp4},
 \begin{equation*}
 \|\bm_{\Delta}(\bx)-\bx'\|\leq\|\bm_{\Delta}(\bx)-\nabla
 g(\bx)\|+\|\nabla g(\bx)-\bx'\|\leq 2\sqrt{\epsilon_1\cdot
 H_g}+\frac{\epsilon_2\cdot d^{\frac{3}{2}}R_{\fF}}{2N}.
 \end{equation*}
Hence by Lemma \ref{lem_ball}, if
 \begin{equation*}
 2\epsilon_3h_gr_{\fF}\geq\epsilon_2\cdot
 d^{\frac{3}{2}}R_{\fF}+4N\sqrt{\epsilon_1\cdot H_g},
 \end{equation*}
then $\bm_{\Delta}(\bx)\in\fF^{(N)}_{g,\epsilon_3}$. Since the later
is convex, we have
$\rvol\left(\Delta^{(N)}\right)\leq\rvol\left(\fF^{(N)}_{g,\epsilon_3}\right)=(1+\epsilon_3)^dh_{11}\cdots
h_{dd}\cdot\rvol\left(\fF^{(N)}\right)$. Now we let
$\epsilon_3=\sqrt[d]{1+\delta}-1$,
$\epsilon_2=\epsilon_3h_gr_{\fF}d^{-\frac{3}{2}}R^{-1}_{\Lambda}$.
Then for a fixed $N\geq N(\epsilon_2)$, when $n$ is large enough and
hence
 \begin{equation*}
 \epsilon_1\leq\frac{1}{H_g}\left(\frac{\epsilon_3h_gr_{\fF}}{4N}\right)^2,
 \end{equation*}
we have
 \begin{equation*}
 \mu_n|_{\fF^{(N)}}(f)\leq(1+\delta)d!\cdot h_{11}\cdots
 h_{dd}\cdot\sup_{\bx\in\fF^{(N)}}f(\bx)\leq(1+\delta)\sup_{\bx\in\fF^{(N)}}f\cdot\rHess_g(\bx)
 \cdot\rvol\left(\fF^{(N)}\right).
 \end{equation*}
Summing over all $(b_1,...,b_d)$, we have
 \begin{equation*}
 \mu_n(f)\leq(1+\delta)\sum_{(b_1,...,b_d)}\sup_{\bx\in\fF^{(N)}_{b_1,...,b_d}}f\cdot\rHess_g(\bx)
 \cdot\rvol\left(\fF^{(N)}_{b_1,...,b_d}\right).
 \end{equation*}
Let $n\rightarrow\infty$ and then $N\rightarrow\infty$, we have
 \begin{equation*}
 \mu(f)\leq(1+\delta)\int_{\fF}f(\bx)\rHess_g(\bx)\rd\bx
 \end{equation*}
which confirms the claim \eqref{claim}.
\end{proof}

\section{A Calabi-Yau theorem}
\label{sec_cy}

In this section, we state and prove the non-archimedean analogue of
the Calabi-Yau theorem for totally degenerate abelian variety $A$.\\

{\em Classical Calabi-Yau theorem review.} Let us have a quick
review of the famous Calabi conjecture which is proved by Yau in
complex geometry. For details, we refer to Yau's original paper
\cite{Ya} and also the book \cite[Chapter 5]{Jo} by Joyce. For
simplicity, we just state it for the algebraic case. Hence let $M$
be a connected compact complex manifold of dimension $d\geq1$ and
$L$ an ample line bundle on it. Given any smooth metric $\|\;\|$ on
$L$, we have the Chern class $\omega=c_1(L,\|\;\|)$ which is a
(smooth) $(1,1)$-form on $M$. It determines a measure, i.e., a top
form $\mu=\omega^{\wedge d}$ on $M$. We say $\|\;\|$ is positive if
$\omega$ is positive definite everywhere. Then the measure $\mu$ is
obviously positive. The Calabi conjecture asserts that given any
smooth positive $(d,d)$-form $\mu'$ such that
$\int_M\mu'=\int_M\mu$, there exists a smooth positive measure
$\|\;\|'$ on $L$, unique up to a scalar, such that
$\mu'=(\omega')^{\wedge d}$ where $\omega'=c_1(L,\|\;\|')$.

If we write $\mu'=\re^f\mu$ for a unique smooth real function $f$ on
$M$, then the Calabi conjecture asserts that there exists a unique
smooth real function $\phi$ such that
\begin{description}
  \item[(1)] $\omega+\rd\rd^c\phi$ is a positive $(1,1)$-form;
  \item[(2)] $\int_M\phi\mu=0$;
  \item[(3)] $(\omega+\rd\rd^c\phi)^{\wedge d}=\re^f\mu$.
\end{description}

If we choose a local coordinates $z_1,...,z_d$ on an open set $U$ in
$M$, then
$\left(g_{\alpha\bar{\beta}}\right)_{\alpha,\bar{\beta}=1,...,d}$ is
an $d\times d$ hermitian matrix, where $g$ is the Riemannian metric
associate with $\omega$. Then the condition (3) reads
\begin{description}
 \item[(3')]
 \begin{equation}\label{cma}
 \det\left(g_{\alpha\bar{\beta}}+\frac{\partial^2\phi}{\partial z_{\alpha}
 \partial\bar{z}_{\bar{\beta}}}\right)=\re^f\det\left(g_{\alpha\bar{\beta}}\right)
 \end{equation}
 which is a complex Monge-Amp\`{e}re equation.
\end{description}
More generally, $\omega$ could just be a K\"{a}hler form. Then the
existence part of the following theorem is due to Yau and the
uniqueness part is due to Calabi.

 \begin{theo}\label{theo_cy}
 Let $M$, $\omega$ be as above, then\\
 (Existence, cf. \cite[\S4, Theorem 1]{Ya}) For any $f\in\sC^k(M)$ ($k\geq3$), there
 exists $\phi\in\sC^{k+1,\alpha}(M)$ for any $\alpha\in[0,1)$ satisfying (1)-(3);\\
 (Uniqueness, cf. \cite{Ca}, \cite[\S5, Theorem 3]{Ya}) For any $f\in\sC^1(M)$, there is at most one
 $\phi\in\sC^3(M)$ satisfying (1)-(3).\\
 \end{theo}

{\em A non-archimedean analogue.} Recall that we have a totally
degenerate abelian variety $A$ of dimension $d$ over $k$ and an
ample line bundle $L$ on it. For any integrable metrized line bundle
$\overline{L}=(L,\|\;\|)$, we define the measure
$c_1(\overline{L})^{\wedge d}$ on the analytic space $A^{\ran}$.
Also, we have a skeleton $i_A:\dR^d/\Lambda\hookrightarrow
A^{\ran}$. The following is a Calabi-Yau theorem in the current
setting for positive measures supported on this skeleton, which has
certain smoothness in the real-analytic sense.

\begin{theo}[\textbf{Non-archimedean Calabi-Yau}]\label{theo_ncy}
Let $A$, $L$, $i_A$, be as above. For any
$\mu\in\sM^k_{>0}(\dR^d/\Lambda)$ ($k\geq3$), there is a
semi-positive metric $\|\;\|$ on $L$, unique up to scalar, such that
$c_1(\overline{L})^{\wedge d}=(i_A)_*\mu$ where
$\overline{L}=(L,\|\;\|)$. Moreover, $\overline{L}$ is toric in the
sense of Definition \ref{defn_toric} whose corresponding Green
function $g$ is in $\sG_+(L)\cap\sC^{k+1,\alpha}(\dR^d)$ for any
$\alpha\in[0,1)$.
\end{theo}

\begin{proof}
The uniqueness part follows from the general theorem on the
uniqueness \cite[Theorem 1.1.1]{YZ} proved by Yuan and Zhang.

Now we prove the existence. Recall that we have a canonical Green
function $g_{\rcan}$ for $L$ which determines a measure
$\mu_{\rcan}$ on $\dR^d/\Lambda$ (which is just $H_q$ times the
Lebesgue measure). By Theorem \ref{theo_measure}, we only need to
prove that for a given $f\in\sC^k(\dR^d/\Lambda)$ ($k\geq3$) such
that $\int_{\dR^d/\Lambda}\re^f\rd\bx=1$, there exists a function
$\phi\in\sC^{k+1,\alpha}(\dR^d/\Lambda)$ for any $\alpha\in[0,1)$
such that:
\begin{itemize}
  \item The matrix
$\left((g_{\rcan})_{ij}+\phi_{ij}\right)_{i,j=1,...,d}$ is positive
definite;
  \item It satisfies the real Monge-Amp\`{e}re equation
  \begin{equation}\label{rma}
  \det\left((g_{\rcan})_{ij}+\frac{\partial^2\phi}{\partial x_i\partial x_j}\right)=\frac{H_q}{d!}\re^f;
  \end{equation}
  \item If $f\in\sC^{\infty}(\dR^d/\Lambda)$, then
$\phi\in\sC^{\infty}(\dR^d/\Lambda)$.
\end{itemize}

We would like to deduce it from the complex case, i.e., theorem
\ref{theo_cy}. We introduce the following manifold
 \begin{equation*}
 \dA=\dR^d/\Lambda\oplus\dR^d/\Lambda
 \end{equation*}
where we write $(x_1,...,x_d;y_1,...,y_d)$ for the usual chart. The
tangent bundle has a canonical splitting
$\cT_{\dA}=\cT_1\oplus\cT_2$ where $\cT_i$ is the pull-back of the
tangent bundle on the $i$-th $\dR^d/\Lambda$. Write $u_i=\partial
x_i$ and $v_i=\partial y_i$ and define a complex structure $J$ on
$\cT_{\dA}$ by $Ju_i=v_i$, $Jv_i=-u_i$ ($i=1,...,d$). Then as a
complex manifold, $\dA$ is isomorphic to
$\dC^d/\Lambda\oplus\Lambda$.

We define
 \begin{equation*}
 \dg\left((u_i,v_j),(u_{i'},v_{j'})\right)=\frac{1}{2}
 \left((g_{\rcan})_{ii'}+(g_{\rcan})_{jj'}\right).
 \end{equation*}
Then $\dg$ is a Riemannian metric on $\dA$ and
$\domega(w,w')=\dg(Jw,w')$ is a K\"{a}hler metric on $(\dA,J)$ with
$\dmu=\domega^{\wedge d}$. Define $\df(\bx,\by)=f(\bx)$ which is in
$\sC^k(\dA)$. Then $\left((\dA,J),\domega,\df\right)$ is in the
situation of Theorem \ref{theo_cy}. Applying this theorem, we see
that there is a unique function $\dphi\in\sC^{k+1,\alpha}(\dA)$ for
any $\alpha\in[0,1)$ satisfying (1)-(3). If we write the
Monge-Amp\`{e}re equation \eqref{cma} explicitly in the current
situation, we see that $\dphi$ satisfies
 \begin{equation}\label{cma1}
 \det\left((g_{\rcan})_{\alpha\bar{\beta}}+\frac{\partial^2\dphi(\bx,\by)}
 {\partial\bz_{\alpha}\partial\bar{\bz}_{\bar{\beta}}}\right)=\frac{H_q}{d!}\re^{f(\bx)}
 \end{equation}
where $\bz_{\alpha}=\bx_{\alpha}+i\by_{\alpha}$ and
$\bar{\bz}_{\bar{\beta}}=\bx_{\bar{\beta}}-i\by_{\bar{\beta}}$. For
any $\by_0\in\dR^d/\Lambda$, let
$\dphi_{\by_0}(\bx,\by)=\dphi(\bx,\by-\by_0)$. Then $\dphi_{\by_0}$
is also a $\sC^{5,\alpha}$-solution satisfying (1)-(3). Hence by the
uniqueness, $\dphi_{\by_0}=\dphi$ for any $\by_0$, i.e.,
 \begin{equation*}
 \frac{\partial\dphi}{\partial\by_i}\equiv0;\qquad i=1,...,d.
 \end{equation*}
Restricting to $\dR^d/\Lambda\times\{\bzero\}$, we see that
$\phi(\bx):=\dphi(\bx,\bzero)\in\sC^{k+1,\alpha}(\dR^d/\Lambda)$ for
any $\alpha\in[0,1)$, satisfies the real Monge-Amp\`{e}re equation
\eqref{rma} and such that $\left((g_{\rcan})_{ij}+\phi_{ij}\right)$
is positive definite. Hence the theorem is proved.
\end{proof}

\begin{rem}
The restriction on the field $k$ is not necessary. In fact, all
results and argument remain valid for any algebraically closed
non-archimedean field whose valuation is non-trivial. One only need
to use the generalized definition of Chambert-Loir's measure given
by Gubler in \cite[\S3]{Gu2}.
\end{rem}

\end{document}